\documentclass[10pt,a4paper]{amsart}
\pdfoutput=1

\usepackage{lmodern}

\usepackage[T1]{fontenc}
\usepackage[utf8]{inputenc}

\usepackage{url}
\usepackage[pagebackref,breaklinks,linktocpage]{hyperref}
\usepackage[capitalise]{cleveref}
\newcommand*{\arxiv}[1]{\href{https://arxiv.org/abs/#1}{arXiv:#1}}
\newcommand*{\doi}[1]{\href{https://doi.org/#1}{doi:#1}}
\usepackage{xifthen}
\newcommand*{\msc}[2][]{\href{https://mathscinet.ams.org/mathscinet/search/mscdoc.html?code=#2,(#1)}{Primary: #2\ifthenelse{\isempty{#1}}{}{; Secondary: #1}.}}

\usepackage{amsmath,amssymb,amsthm,amsfonts,amsbsy,latexsym}
\usepackage[abbrev,alphabetic]{amsrefs}
\usepackage{mathtools}
\usepackage{commath}
\usepackage{nicefrac}

\usepackage{tikz}
\usetikzlibrary{arrows}
\usetikzlibrary{babel}
\usetikzlibrary{cd}
\usetikzlibrary{decorations.pathreplacing}
\usetikzlibrary{matrix}
\usetikzlibrary{quotes}

\makeatletter
\protected\def\tikz@nonactivecolon{\ifmmode\mathrel{\mathop\ordinarycolon}\else:\fi}
\makeatother

\usepackage{enumitem}
\usepackage{csquotes}
\usepackage{verbatim}

\usepackage{etoolbox}
\apptocmd{\sloppy}{\hbadness 10000\relax}{}{}
\usepackage{microtype}
\usepackage{scrhack}

\theoremstyle{definition}
\newtheorem{defi}{Definition}[section]
\crefname{defi}{Definition}{Definitions}

\theoremstyle{plain}
\newtheorem{lemm}[defi]{Lemma}
\crefname{lemm}{Lemma}{Lemmata}
\newtheorem{theo}[defi]{Theorem}
\crefname{theo}{Theorem}{Theorems}
\newtheorem{coro}[defi]{Corollary}
\newtheorem{prop}[defi]{Proposition}

\newtheorem{maintheo}{Theorem}

\crefname{maintheo}{Theorem}{Theorems}

\crefname{lemmenum}{Lemma}{Lemmata}

\theoremstyle{remark}
\newtheorem{rema}[defi]{Remark}
\crefname{rema}{Remark}{Remarks}
\newtheorem{exam}[defi]{Example}

\newtheoremstyle{maintheorem}{}{}{\itshape}{}{\bfseries}{}{.5em}{#1 \!\thmnote{#3}.}
\theoremstyle{maintheorem}

\newcommand*{\ZZ}{\mathbb{Z}}
\newcommand*{\QQ}{\mathbb{Q}}
\newcommand*{\RR}{\mathbb{R}}
\newcommand*{\CC}{\mathbb{C}}

\newcommand*{\F}{\mathfrak{F}}
\newcommand*{\DC}{\mathcal{D}}
\newcommand*{\SC}{\mathcal{S}}
\newcommand*{\EC}{\mathcal{E}}
\newcommand*{\AC}{\mathcal{A}}
\newcommand*{\FC}{\mathcal{F}}

\let\le\leqslant
\let\ge\geqslant
\let\phi\varphi

\DeclareMathOperator{\rk}{rk}

\DeclareMathOperator{\Tor}{Tor}
\DeclareMathOperator{\Ext}{Ext}
\DeclareMathOperator{\Endd}{End}
\DeclareMathOperator{\supp}{supp}
\DeclareMathOperator{\Ore}{Ore}
\DeclareMathOperator{\prd}{pd}
\DeclareMathOperator{\fld}{fd}
\DeclareMathOperator{\Hom}{Hom}
\DeclareMathOperator{\Mat}{Mat}

\DeclareMathOperator{\HC}{\mathcal{H}}
\DeclareMathOperator{\VCyc}{VCyc}
\DeclareMathOperator{\Triv}{Tr}
\DeclareMathOperator{\KB}{\mathbf{K}}
\DeclareMathOperator{\pt}{pt}
\DeclareMathOperator{\Or}{Or}
\DeclareMathOperator{\KBB}{\mathbb{K}^{-\infty}}

\newcommand*{\UC}{\mathcal{U}}
\newcommand*{\NC}{\mathcal{N}}

\begin{document}

\title{Pseudo-Sylvester domains and skew Laurent polynomials over firs}

\author{Fabian Henneke}
\email{\href{mailto:henneke@uni-bonn.de}{henneke@uni-bonn.de}}
\address{Max-Planck-Institut für Mathematik, Vivatsgasse 7, 53111 Bonn, Germany}

\author{Diego L\'opez-\'Alvarez}
\email{\href{mailto:diego.lopez@icmat.es}{diego.lopez@icmat.es}}
\address{Instituto de Ciencias Matem\'aticas (ICMAT), Nicol\'as Cabrera 13-15, Campus de Cantoblanco UAM, 28049 Madrid, Spain}


\begin{abstract}
    Building on recent work of Jaikin-Zapirain, we provide a homological criterion for a ring to be a pseudo-Sylvester domain, that is, to admit a division ring of fractions over which all stably full matrices become invertible.
    We use the criterion to study skew Laurent polynomial rings over free ideal rings (firs).

    As an application of our methods, we prove that crossed products of division rings with free-by-\{infinite cyclic\} and surface groups are pseudo-Sylvester domains unconditionally and Sylvester domains if and only if they admit stably free cancellation. This relies on the recent proof of the Farrell--Jones conjecture for normally poly-free groups and extends previous results of Linnell--Lück and Jaikin-Zapirain on universal localizations and universal fields of fractions of such crossed products.
\end{abstract}

\maketitle
\tableofcontents

\section*{Introduction}

Given a domain $R$, i.e., a not necessarily commutative ring without non-trivial zero divisors, it is natural to ask whether there exists a division ring $\DC$ in which $R$ can be embedded. In the commutative case, the existence of the field of fractions of $R$ settles the question, but in the non-commutative setting, a division ring with the desired property may not exist in general (\cite{Malcev1937}).

It was P.\,M.\,Cohn who realized that, in the same way that we can obtain a field from a commutative ring by localizing at a prime ideal (and then taking the residue field), we can obtain a division ring $\DC$ from any ring $R$ by means of universal localization at prime matrix ideals (cf.\ \cite{Cohn2006}). 
Similarly to the commutative case, the division ring obtained in this way is generated as a division ring by the image of $R$ under the corresponding map $R\to\DC$.
The pair given by $\DC$ and the map $R\to\DC$, or sometimes just $\DC$ if the map is clear from the context, is usually referred to as \emph{epic division $R$-ring}.

Adopting the previous terminology, recall that if $R$ is commutative, an epic field $R\to K$ is completely characterized by its kernel, which is a prime ideal of $R$, in the sense that $K$ can be recovered as mentioned above, i.e., by localizing at the kernel and taking the residue field. This is equivalent to saying that such a homomorphism is determined by the set of elements that become invertible in $K$, the ones outside the kernel. In the very same spirit, P.\,M.\,Cohn showed that a prescribed epic division $R$-ring is completely characterized by its singular kernel, which is a prime matrix ideal, or equivalently by the set $\Sigma$ of matrices becoming invertible under the homomorphism. From the latter point of view, the map will be injective if and only if $\Sigma$ contains every non-zero element of $R$.

Assume that we are given an embedding $R\hookrightarrow \DC$ of the domain $R$ into the division ring $\DC$. Then, a natural necessary condition for an $n\times n$ matrix $A$ over $R$ to become invertible over $\DC$ is that it cannot be expressed as a product $A = BC$ for some matrices $B$, $C$ of sizes $n\times m$ and $m\times n$, respectively, where $m<n$. Otherwise, the usual rank $\rk_{\DC}(A)$ of $A$ over $\DC$ would be less or equal than $m$, and hence $A$ would not be invertible. A matrix satisfying this necessary condition is called \emph{full}. Therefore, one may wonder whether, among the division rings in which $R$ can be embedded, there exists one in which we can invert every full matrix. The rings for which this is possible, originally studied by W. Dicks and E. Sontag (\cite{DS1978}) as those satisfying the law of nullity with respect to the inner rank function, comprise the family of \emph{Sylvester domains}. The first examples of Sylvester domains were the free ideal rings (firs) (cf.\ \cite{Cohn2006}*{Section~5.5}).

In addition, observe that if the matrix $A$ is to become invertible in a division ring, then the same holds true for $A\oplus I_m$, the block diagonal matrix with blocks $A$ and $I_m$, where $I_m$ denotes the $m\times m$ identity matrix. Thus, $A\oplus I_m$ must in fact be full for every non-negative integer $m$. A matrix with this property is called \emph{stably full} and, of course, in a Sylvester domain it is the case that every full matrix is stably full. Nevertheless, in general there may be full matrices that are not stably full, and hence, the question of whether there exists a division ring $\DC$ in which $R$ embeds and in which we can invert every stably full matrix over $R$ is interesting in its own right. The rings with this property are the \emph{pseudo-Sylvester domains}, which were introduced in \cite{CS1982} as the family of stably finite rings satisfying the law of nullity with respect to the stable rank function. 

Notice that in both of the scenarios described in the two previous paragraphs, if such a division ring $\DC$ exists, then it is necessarily \emph{universal} in the sense of P.\,M.\,Cohn (see \cref{subsect:localization}), meaning that if a matrix $A$ over $R$ becomes invertible over some division ring, then it is also invertible over $\DC$.

Recently, in \cite{Jaikin2019_sylvester}, A. Jaikin-Zapirain introduced a new homological criterion for a ring to be a Sylvester domain. In \cref{sect:Theorem_A}, we provide a similar recognition principle for pseudo-Sylvester domains and we use it to prove the following result:

\begin{maintheo}
    \label{theo:main_fir_crossed_ZZ}
    Let $\F$ be a fir with universal division $\F$-ring of fractions $\DC_{\F}$, and consider a crossed product ring $\SC=\F\ast \ZZ$. Then, the following holds:
      \begin{enumerate}
          \item[a)] $\SC$ is a pseudo-Sylvester domain if and only if every finitely generated projective $\SC$-module is stably free.
          \item[b)] $\SC$ is a Sylvester domain if and only if it is projective-free.
      \end{enumerate}
    In any of the previous situations, the crossed product $\F\ast \ZZ$ can be extended to a crossed product $\DC_{\F}\ast \ZZ$ and $\DC_{\SC} = \Ore(\DC_{\F}\ast \ZZ)$, the Ore division ring of fractions of $\DC_{\F}\ast \ZZ$, is the universal division $\SC$-ring of fractions. Furthermore, it is isomorphic to the universal localization of $\SC$ with respect to the set of all stably full (resp.\ full) matrices.
\end{maintheo}
As a particular application of \cref{theo:main_fir_crossed_ZZ}, we obtain the next result in \cref{sect:applications} through the recent advances on the Farrell--Jones conjecture by Bestvina--Fujiwara--Wigglesworth and Brück--Kielak--Wu:

\begin{maintheo}
    \label{theo:main_pseudosylvester}
    Let $E$ be a division ring and $G$ a group arising as an extension
    \[1\to F\to G\to \ZZ \to 1\]
    where $F$ is a free group.
    Then every crossed product $E\ast G$ is a pseudo-Sylvester domain.
    In particular, $\DC_{E\ast G}=\Ore(\DC_{E\ast F}\ast \ZZ)$ is the universal division $E\ast G$-ring of fractions and it is isomorphic to the universal localization of $E\ast G$ with respect to the set of all stably full matrices.
    Moreover, $E\ast G$ is a Sylvester domain if and only if it has stably free cancellation.
\end{maintheo}

Some examples of groups to which \cref{theo:main_pseudosylvester} can be applied and for which it is known whether they admit stably free cancellation are given in \cref{subsect:examples}.

Note that Jaikin-Zapirain already showed in \cite{Jaikin2019_hughes}*{Theorem~1.1 \& Theorem~3.7} that $E\ast G$ has a universal division ring of fractions.
Furthermore, in \cite{LL2018}*{Theorem~2.17}, it has already been shown that $KG$, where $K$ is a subfield of $\CC$, admits a universal localization that is a division ring.
\cref{theo:main_pseudosylvester} implies both of these results and additionally provides a description of the matrices that become invertible over $\DC_{E\ast G}$.

The universal division $E\ast G$-rings shown to exist above can actually be identified with previously constructed division rings using Hughes-freeness, which will be the topic of \cref{subsect:locally_indicable}.
First, observing that the group $G$ is locally indicable and thus admits a Conradian order $\leq$, results of Gräter \cite{Grater2019} imply that $\DC_{E\ast G}$ is the division closure of $E\ast G$ within the endomorphism ring of the $E$-vector space $E((G,\leq))$ of Malcev--Neumann series.
Second, for the particular case of a classical group ring $KG$ with a subfield $K$ of $\CC$, the division ring $\DC_{KG}$ coincides with the Linnell division ring $\DC(G;K)$ constructed by Linnell \cite{Linnell1993} in his proof of the Atiyah conjecture for a class of groups containing the groups considered in \cref{theo:main_pseudosylvester}.

\subsection*{Acknowledgements}
The authors want to thank Andrei Jaikin-Zapirain for suggesting to use his recent results on Sylvester domains to study the group rings of surface groups.
They are also thankful to Andrei Jaikin-Zapirain, Daniel Kasprowski, and Wolfgang Lück for helpful discussions.

The present work is part of the authors' PhD projects at the University of Bonn and the Autonomous University of Madrid (UAM), respectively.

The first author was supported by Wolfgang Lück's ERC Advanced Grant ``KL2MG-interactions'' (no.\ 662400) granted by the European Research Council, as well as by the Max Planck Institute for Mathematics.

The second author was partially supported by the grant MTM2017-82690-P of the Spanish MINECO, and by the ICMAT ``Severo Ochoa'' project SEV-2015-0554. The second author would also like to thank the Max Planck Institute for Mathematics for the hospitality and support, since the foundations for this paper were laid while he was a guest at this institute.

\section{Definitions and background} \label{sect:background}

All rings are assumed to be associative with unit, but not necessarily commutative. If not otherwise specified, modules are taken to be left modules.

\smallskip

\subsection{Crossed products}\label{subsect:crossed_products}
Let $R$ be a ring and $G$ a group.
A \emph{crossed product} $S$ is a ring that contains $R$ and a set of units $\tilde G=\{\tilde g\mid g\in G\}$, a copy of $G$ as a set, such that:
\begin{itemize}
    \item as a left $R$-module, $S$ is free with basis $\tilde G$, i.e., $S = \bigoplus_{g\in G} R\tilde g$, and $\tilde e = 1_S$;
    \item for every $\tilde g,\tilde h\in G$, we have $R\tilde g = \tilde g R$ and $R\tilde g\tilde h = R\widetilde{gh}$.
\end{itemize}
In other words, $S$ is a $G$-graded ring with base ring $R$ that has a unit in each component. Note that it also follows that $S$ is free as a right $R$-module. We denote such a ring $S$ simply by $R*G$, but caution the reader that the full ring structure is not reflected in the notation.

For every subgroup $N$ of $G$ there exists a subring $R*N$ of $R\ast G$, and if $N$ is normal in $G$ and $G/N$ is infinite cyclic, then we have $R\ast G = (R\ast N)\ast G/N \cong (R\ast N)[t^{\pm 1};\tau]$ for some automorphism $\tau$ of $R\ast N$, where the right hand side denotes the skew Laurent polynomial ring with coefficients in $R\ast N$ and multiplication rule $tx=\tau(x)t$ for all $x\in R\ast N$ (cf.\ \cite{MR2001}*{1.5.9~Lemma~(ii) \& 1.5.11~Proposition~(i)}). In particular, every crossed product $R\ast \ZZ$ is isomorphic to a skew Laurent polynomial ring $R[t^{\pm 1};\tau]$ and every such ring is an instance of crossed product with $\ZZ$.

\subsection{Ore and universal localizations} \label{subsect:localization}
Recall that whenever we have a commutative domain $R$, we can consider its field of fractions, a field in which every element can be expressed as a fraction of the form $rs^{-1}$ for some $r,s\in R$ where $s$ is non-zero. When $R$ is a non-commutative domain, a division ring with such a description may not exist in general, since we have no way in principle to ensure that sums and products of elements of the form $rs^{-1}$ admit a similar expression. The condition to ensure the feasibility of this procedure is the so-called Ore condition.

Assume that we are given a ring $R$, non necessarily a domain for the moment, and let $T$ be a multiplicative set of non-zero-divisors in $R$. We say that $T$ is \emph{right Ore} if, for every $r\in R, t\in T$, there exist $s\in R, u\in T$ such that $ru = ts$. Under this condition, one can construct the \emph{right Ore localization} $\Ore_{r,T}(R) = RT^{-1}$, a ring whose elements can be written as $rt^{-1}$ for $r\in R, t\in T$ (cf.\ \cite{GW2004}*{Theorem~6.2}). Moreover, the ring $\Ore_{r,T}(R)$ is flat as a left $R$-module (cf.\ \cite{GW2004}*{Corollary~10.13}), i.e., the tensor $\square \otimes_R \Ore_{r,T}(R)$ preserves short exact sequences of right $R$-modules. Similarly we can define left Ore sets, and if $T$ is both left and right Ore, then $\Ore_{l,T}(R) = \Ore_{r,T}(R)$ (cf.\ \cite{GW2004}*{Proposition~6.5}).

Thus, observe that if $R$ is a \emph{right Ore domain}, i.e., if it is a domain and the set $T = R\backslash \{0\}$ of all non-zero elements of $R$ is right Ore, the right Ore localization $\Ore_r(R):= \Ore_{r,T}(R)$ is a division ring whose elements are fractions of the previous form (cf.\ \cite{GW2004}*{Theorem~6.8}). If $R$ is a right and left Ore domain, we just say that it is an \emph{Ore domain}, and denote its Ore localization by $\Ore(R)$. For instance, this is the case of a skew Laurent polynomial ring of the form $R[t^{\pm 1};\tau]$ where $R$ is both a right and a left Noetherian domain and $\tau$ is an automorphism of $R$ (cf.\ \cite{GW2004}*{Corollary~1.15 \& Corollary~6.7})).\medskip

Going a step further, one can consider the general question of whether a given non-commutative domain $R$ can be embedded at all into a division ring. In this full generality, it can be treated by means of P.\,M.\,Cohn's theory of epic division $R$-rings (cf.\ \cite{Cohn2006}*{Chapter~7}), which relies on the existence of prime matrix ideals (for the definition of this notion, we refer the reader to \cite{Cohn2006}*{Chapter~7, Section~3}), and universal localizations.
\begin{defi}
 Given a set $\Sigma$ of (square) matrices over $R$, and a homomorphism of rings $\varphi\colon R \to S$, we say that the map $\varphi$ is \emph{$\Sigma$-inverting} if every element of $\Sigma$ becomes invertible over $S$. We say that $\varphi$ is \emph{universal $\Sigma$-inverting} if any other $\Sigma$-inverting homomorphism factors uniquely through $\varphi$. In this latter case, we denote $S = R_{\Sigma}$ and we call $R_{\Sigma}$ the universal localization of $R$ with respect to $\Sigma$.
\end{defi}
If we allow $R_{\Sigma}$ to be the zero ring, the existence of the universal localization can always be proved by taking a presentation of $R$ as a ring and formally adding the necessary generators and relations. Moreover, the universal $\Sigma$-inverting homomorphism will be injective if and only if there exists a $\Sigma$-inverting embedding to some ring (\cite{Cohn2006}*{Theorem~7.2.4}).

To briefly explain P.\,M.\,Cohn's main result on the topic, we need to introduce the notion of epic division $R$-ring.
\begin{defi}
    Given a ring $R$, an \emph{epic division $R$-ring} is a division ring $\DC$ together with a ring homomorphism $R\to \DC$ such that $\DC$ is generated, as a division ring, by the image of $R$. If additionally the homomorphism is injective, then we say that $R\hookrightarrow \DC$ (or simply $\DC$) is a \emph{division $R$-ring of fractions}.
\end{defi}
The ``epic'' terminology is justified through the fact that $\DC$ being generated by the image of $R$ is equivalent to the homomorphism being an epimorphism in the category of rings (\cite{Cohn2006}*{Corollary 7.2.2}). With this, P.\,M.\,Cohn proved that epic division $R$-rings are completely characterized (up to $R$-isomorphism) by the set $\Sigma$ of matrices over $R$ that become invertible in the division ring, and that they always arise as residue fields of a universal localization $R_{\Sigma}$ (\cite{Cohn2006}*{Theorem~7.2.5 \& Theorem~7.2.7}). In addition, such sets $\Sigma$ are precisely the complements in the set of square matrices over $R$ of prime matrix ideals $\mathcal P$ (\cite{Cohn2006}*{Theorem~7.4.3}). Thus, we would have an embedding of $R$ into a division ring if we can construct such a set $\Sigma$ including all non-zero elements in $R$.

Finally, if among all the possible division $R$-rings of fractions there exists one in which we can invert ``the most'' (relative to $R$) matrices possible, we call it the universal division $R$-ring of fractions. More precisely:

\begin{defi}
 The division $R$-ring of fractions $R\hookrightarrow \DC$  is called \emph{the universal division $R$-ring of fractions} if, for any other epic division $R$-ring $\DC'$, the set $\Sigma'$ of matrices that become invertible over $\DC'$ is contained in the set $\Sigma$ of matrices that become invertible over $\DC$.
\end{defi}

In \cref{subsect:Sylvester} we will introduce two families of rings, namely Sylvester and pseudo-Sylvester domains, for which there exists a universal division ring of fractions and for which the set $\Sigma$ of matrices becoming invertible under the embedding can be characterized in a natural way only depending on $R$. Our main result will be to build on a homological criterion for a ring to belong to one of these families, which is why we need to introduce parts of the dimension theory of (non-commutative) rings in the following.

\subsection{Weak and global dimensions} \label{subsect:dimensions}

Recall that a module $N$ over a ring $R$ has \emph{projective dimension} at most $n$ (abbreviated $\prd(N)\le n$) if $N$ admits a resolution
\[0\to P_n\to\ldots\to P_0\to N\to 0\]
of projective $R$-modules. In particular, observe that $N$ is projective if and only if $\prd(N) = 0$. The supremum among the projective dimensions of all left (resp.\ right) $R$-modules is called the \emph{left} (resp.\ \emph{right}) \emph{global dimension} of $R$, and it is not left-right symmetric in general. This concept is deeply related to $\Ext$ functors.

\begin{lemm}[{\cite{Rotman2009}*{Proposition 8.6}}]
    \label{lemm:ext_rotman}
    Let $N$ be a left $R$-module. Then $\prd(N)\le n$ if and only if $\Ext_R^{n+1}(N, N')=0$ for all left $R$-modules $N'$.
\end{lemm}

Analogously, we say that the \emph{flat dimension} of $N$ is at most $n$, and we write $\fld(N)\le n$, if it admits a resolution of flat $R$-modules
\[0\to Q_n\to\ldots\to Q_0\to N\to 0,\]
and define the \emph{left} (resp.\ \emph{right}) \emph{weak dimension} of $R$ as the supremum of the flat dimensions of all left (resp.\ right) $R$-modules. It turns out that this notion is always left-right symmetric (\cite{Rotman2009}*{Theorem 8.19}) and hence we can just talk about the \emph{weak dimension} of $R$. As it happens with $\prd(N)$ and $\Ext^\ast_R(N,\square)$, the flat dimension of $N$ (resp.\ of a right $R$-module $M$) can be characterized in terms of $\Tor_\ast^R(\square,N)$ (resp.\ $\Tor_\ast^R(M,\square)$). Observe though that, unlike the previous case, here we need to change the argument while considering left or right modules.

\begin{lemm}[{\cite{Rotman2009}*{Proposition 8.17}}]
    \label{lemm:tor_rotman}
    Let $N$ be a left $R$-module. Then $\fld(N)\le n$ if and only if $\Tor^R_{n+1}(M, N)=0$ for all right $R$-modules $M$.
\end{lemm}

We finish this section with the following result regarding $\Tor$, sometimes referred to as Shapiro's Lemma. A generalization proved using spectral sequences can be found in \cite{Rotman2009}*{Corollary~10.61}, but we give an elementary proof for the precise statement we need.

\begin{lemm}
 \label{lemm:shapiro}
 Let $R$ be a subring of $S$ such that $S$ is flat as a left $R$-module. Then, for any right $R$-module $M$, for any left $S$-module $N$ and for any $n\ge 0$, we have
 \[
  \Tor_n^R(M,{_R}N) \cong \Tor_n^S(M\otimes_R S, N)
 \]
 where ${_R}N$ denotes $N$ considered as a left $R$-module.
\end{lemm}

\begin{proof}
    Assume that we have a projective resolution for $M$
    \[
     \ldots\to P_k\to\ldots\to P_0\to M\to 0.
    \]
    Since $S$ is a flat left $R$-module, the following sequence is also exact of projective right $S$-modules, i.e., a projective resolution for $M\otimes_R S$
    \[
     \ldots\to P_k\otimes_R S\to\ldots\to P_0\otimes_R S\to M\otimes_R S\to 0.
    \]
    Now, just observe that computing $\Tor_\ast^R(M, {_R}N)$ amounts to computing the homology of the chain
    \[
     \ldots\to P_k\otimes_R N\to\ldots\to P_0\otimes_R N\to 0
    \]
    and that computing $\Tor_\ast^{S}(M\otimes_R S, N)$ amounts to computing the homology of
    \[
     \ldots\to P_k\otimes_R S \otimes_S N \to\ldots\to P_0\otimes_R S \otimes_S N\to 0
    \]
    Since $S\otimes_S N \cong N$, the result follows.
\end{proof}

\subsection{Stably freeness and stably finite rings}\label{subsect:stably_freeness}

The criteria we are going to introduce in \cref{sect:Theorem_A} rely on proving that certain submodules are finitely generated free or stably free, respectively. Therefore, we need to deal with the latter concept and its relation with the notion of stably finite ring.

\begin{defi}
    A module $M$ over a ring $R$ is called \emph{stably free} if there exists $n \ge 0$ such that $M\oplus R^n$ is a free $R$-module.
\end{defi}

By a result of Gabel, a proof of which is given in \cite{Lam1978}*{Proposition~4.2}, any stably free module that is not finitely generated is already free.
For this reason,  we will restrict our attention to finitely generated stably free modules in the following.

If $M$ is a finitely generated stably free $R$-module and $M\oplus R^n$ is free, then this free module is necessarily finitely generated and hence isomorphic to some $R^m$.
In general, the difference $m-n$ needs neither be positive nor uniquely determined by $M$, reason why we introduce the stably finite property. \medskip

Recall that a ring $R$ is said to be \emph{stably finite} (or \emph{weakly finite}) if whenever $A$ and $B$ are two $n\times n$-matrices over $R$ such that $AB=I_n$, then also $BA=I_n$. This can be reformulated in terms of modules by saying that if $R^n\oplus K\cong R^n$, then $K = 0$.
For example, every division ring is stably finite. Also, if $K$ is a field of characteristic 0 and $G$ is any group, or if $K$ has positive characteristic and $G$ is sofic, the group ring $KG$ is stably finite (cf.\ \cite{Jaikin2019_survey}*{Corollary 13.7}).
Furthermore, any subring of a stably finite ring is again stably finite.

If $M$ is a non-trivial module over a stably finite ring $R$ and $M\oplus R^n\cong R^m$, then the difference $m-n$ is positive and constant among all such representations. We call this positive number the \emph{stably free rank} of $M$ and denote it by $\rk_{sf}(M)$.\medskip

To finish this subsection, let $P$ be a finitely generated projective module over $R$. We will recall in the next subsection that if $R$ is a Sylvester domain then $P$ is necessarily free, while if $R$ is just a pseudo-Sylvester domain, we can only deduce that $P$ is stably free. Thus, a first (and in fact, the only) obstruction for a pseudo-Sylvester domain to be a Sylvester domain is the following property.

\begin{defi}
    A stably finite ring $R$ is said to have \emph{stably free cancellation (SFC)} if every finitely generated stably free $R$-module $M$ is free of rank $\rk_{sf}(M)$.
\end{defi}

In \cite{Cohn2006}, rings with invariant basis number (IBN) and stably free cancellation are called \emph{Hermite rings}. Here, we keep the terminology ``stably free cancellation'' because Hermite rings have also been given other meanings in the literature.

Examples of group rings with and without stably free cancellation will be given in \cref{subsect:examples}.

\begin{rema}
    \label{rema:left_right}
    Let $R$ be a ring. If $M$ is a left (right) $R$-module, then $M^*\coloneqq \Hom_R(M, R)$, called the \emph{dual of $M$}, is naturally a right (left) $R$-module.
    The functor $\Hom_R(\square, R)$, sending $P$ to $P^*$, defines an equivalence between the category of finitely generated projective left $R$-modules and the opposite of the category of finitely generated projective right $R$-modules, with the inverse functor given in the same way.
    To see that $P\cong P^{**}$ for a finitely generated projective $R$-module $P$, note that taking the dual commutes with finite direct sums and the claim thus needs to be checked only for $R$ itself viewed as an $R$-module, where it is clear.
    The equivalence defined in this way restricts to equivalences of the respective subcategories of finitely generated stably free and finitely generated free modules.

    As a consequence, every property of rings that can be expressed in terms of these categories in a way that is invariant under passing to an equivalent or opposite category will hold for left modules if and only if it holds for right modules.
    In particular, whether or not any of the classes of finitely generated projective, stably free or free modules coincide for a ring does not depend on whether left or right modules are considered.
\end{rema}

\subsection{(Pseudo-)Sylvester domains} \label{subsect:Sylvester}
In this section we introduce the main families of rings we are going to deal with throughout the paper, namely, Sylvester domains and pseudo-Sylvester domains. We are going to define them in terms of inner and stable rank of matrices, what a priori may seem to be unrelated to the topics discussed in the previous subsections, but we will see how they interact.

Let $R$ be a ring, and $A$ an $m\times n$ matrix over $R$. Recall that the \emph{inner rank} $\rho(A)$ is defined as the least $k$ such that $A$ admits a decomposition $A = B_{m\times k}C_{k\times n}$. We say that a square matrix $A$ of size  $n\times n$  is \emph{full} if $\rho(A) = n$. Recall also that the \emph{stable rank} $\rho^\ast(A)$ is given by
\[\rho^\ast (A) = \lim_{s\to \infty} \left[\rho(A\oplus I_s)-s\right], \]
whenever the limit exists, where $A\oplus I_s$ denotes the block diagonal matrix with blocks $A$ and $I_s$. We analogously say that a square matrix is \emph{stably full} if it has maximum stable rank. When $R$ is stably finite, $\rho^\ast(A)$ is well-defined and non-negative, and it is positive if $A$ is a non-zero matrix (\cite{Cohn2006}*{Proposition 0.1.3}). For this reason, in the following we restrict our attention to stably finite rings.

Observe that from the definition of the inner rank it follows that the sequence in the limit is always non-increasing and bounded above by $\rho(A)$. In particular, for an $n\times n$ matrix $A$ we obtain that $\rho^\ast(A)\le\rho(A)\le n$ and that $\rho^\ast(A)=n$ if and only if the sequence is constantly $n$. Thus, $A$ is stably full if and only if $\rho(A\oplus I_s) = n+s$ for every $s\ge 0$.

We summarize useful properties of the stable rank over stably finite rings.
\begin{lemm} \label{lemm:stable_rank}
Let $R$ be a stably finite ring. Then the following holds for every matrix $A$ over $R$:
\begin{enumerate}[label={(\arabic*)},ref={\thetheo~(\arabic*)}]
    \item \label[lemmenum]{lemm:stable_rank:additive}For every $k\ge 0$, $\rho^\ast(A\oplus I_k) = \rho^\ast(A)+k$.
    \item \label[lemmenum]{lemm:stable_rank:inner_rank} There exists $N\ge 0$ such that for every $l\ge N$, $\rho^\ast(A\oplus I_l) = \rho(A\oplus I_l)$.
    \item \label[lemmenum]{lemm:stable_rank:inequality}$0\le \rho^\ast(A)\le \rho(A)$.
\end{enumerate}
\end{lemm}

\begin{proof}
  Since $R$ is stably finite, we know that $\rho^\ast(A) = r \ge 0$. This means that there exists $N\ge 0$ such that for any $l\ge N$ we have $\rho(A\oplus I_l)-l = r$. Thus, for $k\ge 0$,
     \[ \rho^\ast(A\oplus I_k) = \lim_{s\to \infty} \left[ \rho(A\oplus I_k \oplus I_s)-(s+k)+k\right] = r + k = \rho^\ast(A)+k. \]
  From here, we also deduce that for $l\ge N$ one has
  \[ \rho(A\oplus I_l) = l + r = l+\rho^\ast(A) = \rho^\ast(A\oplus I_l).\]
  The last statement has already been observed above.
\end{proof}

We can now introduce the main notions of the subsection. Let us define first the notion of Sylvester domain, together with the main examples and properties.

\begin{defi}
  A non-zero ring $R$ is a \emph{Sylvester domain} if $R$ is stably finite and satisfies the law of nullity with respect to the inner rank, i.e., if $A\in \Mat_{m\times n}(R)$ and $B\in \Mat_{n\times k}(R)$ are such that $AB = 0$, then
  \[\rho(A)+\rho(B) \le n. \]
\end{defi}

In fact, it can be shown that the condition that $R$ is stably finite is redundant here, but we keep it as a requirement to show the symmetry with the upcoming definition of pseudo-Sylvester domain.

The following rings serve as the most prominent examples of Sylvester domains (\cite{Cohn2006}*{Proposition~5.5.1}):

\begin{defi}
A \emph{free ideal ring (fir)} is a ring in which every left and every right ideal is free of unique rank (as a module).
\end{defi}

As a consequence, in a fir every submodule of a free module is again free (cf.\ [Coh06, Corollary 2.1.2] and note that every submodule of a free $R$-module of rank $\kappa$ is $\max(|R|, \kappa)$-generated). For instance, a division ring $\DC$ is a fir, and the inner rank over $\DC$ is just its usual rank, which will be denote by $\rk_{\DC}$.
An important example is the group ring $KF$, where $K$ is a field and $F$ is a free group. This result was originally proved by P.\,M.\,Cohn, and we refer the reader to \cite{Lewin1969}*{Theorem~1} for a concise treatment.
More generally, for any division ring $E$ and free group $F$, the crossed product $E\ast F$ is a fir.
This is a consequence of Bergman's coproduct theorem (cf.\ \cite{Sanchez2008}*{Theorem~4.22 (i)}).

The following property of a ring, which by \cref{rema:left_right} is left-right symmetric, is intimately related to Sylvester domains.

\begin{defi}
    A ring $R$ is called \emph{projective-free} if every finitely generated projective $R$-module is free of unique rank.
\end{defi}

Note, for instance, that if $K$ is a field, then the polynomial ring $K[t_1,\dots,t_n]$ in $n$ indeterminates is projective-free, a result known as the Quillen-Suslin theorem.

Every Sylvester domain is projective-free and has weak dimension at most 2 (cf.\ \cite{DS1978}*{Theorem~6} and the subsequent discussion). In \cref{theo:main_fir_crossed_ZZ}, we will provide a class of rings of weak dimension at most 2 which are Sylvester domains if and only if they are projective-free.\medskip

In the same way that Sylvester domains are defined in terms of inner rank, pseudo-Sylvester domains are defined in terms of stable rank.

\begin{defi}
    A non-zero ring $R$ is a \emph{pseudo-Sylvester domain} if $R$ is stably finite and satisfies the law of nullity with respect to the stable rank, i.e., if $A\in \Mat_{m\times n}(R)$ and $B\in \Mat_{n\times k}(R)$ are such that $AB = 0$, then
  \[ \rho^\ast(A)+\rho^\ast(B)\le n.\]
\end{defi}

\begin{exam}
    The following rings are pseudo-Sylvester domains, but not Sylvester domains:
    \begin{itemize}
        \item The polynomial ring $D[x, y]$ in two variables over a division ring $D$ is a pseudo-Sylvester domain by \cite{CS1982}*{Proposition~6.5} and \cite{Bass1968}*{Theorem~XII.3.1}. It is not projective-free by \cite{OS1971}*{Proposition~1} if $D$ is non-commutative.
        \item The Weyl algebra $A_1(K)$ for a field $K$, which is the quotient of the free algebra on two generators $x$ and $y$ by the ideal generated by $xy-yx-1$, is a pseudo-Sylvester domain by \cite{CS1982}*{Proposition~6.5} and \cite{Stafford1977b}*{Theorem~2.2}.
        An example of a projective non-free ideal is provided in \cite{Stafford1977}*{Section~6}.
    \end{itemize}
\end{exam}

In analogy to the case of Sylvester domains, any finitely generated projective module over a pseudo-Sylvester domain is stably free \cite{Cohn2006}*{Proposition 5.6.2}. Moreover, a pseudo-Sylvester domain is a Sylvester domain if and only if the ring has stably free cancellation by \cite{CS1982}*{Proposition~6.1}. \medskip

Several characterizations of Sylvester and pseudo-Sylvester domains can be found in \cite{Cohn2006}*{Theorem~7.5.13} and \cite{Cohn2006}*{Theorem 7.5.18}, respectively. In particular, they can be defined in terms of universal localizations and universal division rings of fractions. In this flavour, observe that for an $n\times n$ matrix $A$ to become invertible over a division ring $\DC$, we need $A$ to be stably full, since otherwise there would exist $s\ge 0$ such that $\rho(A\oplus I_s)<n+s$ and hence $A\oplus I_s$ would not be invertible over $\DC$. Thus, one can wonder whether there exists a division ring in which $R$ embeds and in which every stably full matrix can be inverted. The family of rings for which this is possible is precisely the family of pseudo-Sylvester domains.

For a Sylvester domain, the inner rank is additive, in the sense that $\rho(A\oplus B) = \rho(A)+\rho(B)$ holds for any matrices $A$ and $B$ (cf.\ \cite{Cohn2006}*{Lemma 5.5.3}), and thus the inner and stable rank coincide. Indeed, if $\rho^{\ast}(A) = r$, then by \cref{lemm:stable_rank:inner_rank} there exists $s\ge 0$ such that $\rho(A\oplus I_s) = \rho^{\ast}(A\oplus I_s)$, from where \cref{lemm:stable_rank:additive} and additivity tell us that $\rho^{\ast}(A) = \rho(A)$. As a consequence, every full matrix is actually stably full, and hence Sylvester domains will form the family of rings embeddable into a division ring in which we can invert all full matrices.
We record this in the following proposition, whose statement is implicit in \cite{Cohn2006}*{Theorem~7.5.13 \& Theorem~7.5.18}.

\begin{prop} \label{prop:equivalence}
For a non-zero ring $R$, the following are equivalent:
  \begin{enumerate}
  \item $R$ is a Sylvester (resp.\ pseudo-Sylvester) domain.
  \item There exists a division $R$-ring of fractions $R\hookrightarrow\DC$ such that every full (resp.\ stably full) matrix over $R$ becomes invertible over $\DC$.
  \end{enumerate}
 Moreover, if $R$ satisfies one, and hence each of the previous properties, $\DC$ is the universal division $R$-ring of fractions, and it is isomorphic to the universal localization of $R$ with respect to the set of all full (resp.\ stably full) matrices over $R$.
\end{prop}

\section{Proof of \texorpdfstring{\cref{theo:main_fir_crossed_ZZ}}{Theorem A}} \label{sect:Theorem_A}

This section is devoted to prove \cref{theo:main_fir_crossed_ZZ} by verifying the conditions of \cref{theo:sylvester_criterion,theo:pseudosylvester_criterion}, both of which will be stated in \cref{subsect:criteria}.
The former is a particular case of a homological criterion introduced by Jaikin-Zapirain in \cite{Jaikin2019_sylvester} to determine when a ring with a prescribed embedding into a division ring is a Sylvester domain.
The latter is an adaptation of this recognition principle to pseudo-Sylvester domains.

Throughout this section, $\F$ will always denote a fir with universal division $\F$-ring of fractions $\DC_{\F}$, and we will consider any crossed product ring $\SC = \F \ast \ZZ$.

The following lemma tells us in particular that the crossed product $\SC = \F\ast \ZZ$ can be embedded in a crossed product $\DC_{\F}\ast \ZZ$, and that this ring is an Ore domain.

\begin{lemm}
\label{lemm:tensor_vs_crossed_product2}
    Let $R$ be a (pseudo-)Sylvester domain with universal division $R$-ring of fractions $\DC_R$ and let $R\ast \ZZ$ be a crossed product. Then we can form a crossed product $\DC_R \ast \ZZ$ together with an embedding $R\ast \ZZ \hookrightarrow \DC_R\ast \ZZ$ such that
    \begin{itemize}
        \item[(i)] The composition $R\ast \ZZ \hookrightarrow \DC_R\ast \ZZ \hookrightarrow \Ore(\DC_R \ast \ZZ)$ is a division $R\ast \ZZ$-ring of fractions.
        
        \item[(ii)] The left $R\ast \ZZ$-modules $\DC_R \ast \ZZ$ and $(R\ast \ZZ) \otimes_{R} \DC_R$ (resp. the right  $R\ast \ZZ$-modules $\DC_R \ast \ZZ$ and $\DC_R \otimes_R (R\ast \ZZ)$) are isomorphic.
    \end{itemize}
\end{lemm}

\begin{proof}
    First, we are going to see that every automorphism $\varphi$ of $R$ extends uniquely to an automorphism of $\DC_R$. Indeed, let $\Sigma$ denote the set of (stably) full matrices over $R$ and notice that $\varphi$ preserves $\Sigma$ (i.e., $\varphi(\Sigma) = \Sigma$). Thus, the composition $R\xrightarrow{\varphi} R \hookrightarrow \DC_R$ is a $\Sigma$-inverting embedding, and hence the universal property of universal localization gives us a unique injective map $\varphi\colon R_{\Sigma} = \DC_R \to \DC_R$ such that the diagram
    \[
   \begin{tikzcd}
    R \arrow[d,"\varphi", "\cong"'] \arrow[r,hook] & \DC_R \arrow[d,dotted,"\tilde \varphi"] \\
    R \arrow[r] \arrow[r,hook] & \DC_R.
   \end{tikzcd}
\]
    commutes. Since $\DC_R$ is generated by $R$ as a division ring, $\tilde \varphi$ is also surjective, and hence an automorphism of $\DC_R$.

    As mentioned in \cref{subsect:crossed_products}, we have a ring isomorphism $R\ast \ZZ \cong R[t^{\pm 1};\tau]$ for some automorphism $\tau$ of $R$, and taking the automorphism $\tilde \tau$ of $\DC_R$ that extends $\tau$, we can form the ring $\DC_{R}[t^{\pm 1};\tilde \tau]$, so that we have a commutative diagram
    \[
   \begin{tikzcd}
    R \arrow[d,hook] \arrow[r,hook] & \DC_R  \arrow[d,hook] \\
    R[t^{\pm 1};\tau] \arrow[r] \arrow[r,hook] & \DC_R[t^{\pm 1};\tilde \tau].
   \end{tikzcd}
\]
   To see that the bottom map is epic, let $S$ be any ring and $f,g\colon \DC_R[t^{\pm 1};\tilde \tau]\rightarrow S$ ring homomorphisms that agree on $R[t^{\pm 1};\tau]$. They induce ring homomorphisms $\DC_R\to S$ that coincide on $R$, and hence, since the embedding $R\hookrightarrow \DC_R$ is epic, $f$ and $g$ agree on $\DC_R$. Since they also agree in the indeterminate, we deduce that $f = g$.
    
   Thus, we have an epic embedding $R\ast \ZZ\cong R[t^{\pm 1};\tau] \hookrightarrow \DC_R[t^{\pm 1};\tilde \tau] = \DC_R*\ZZ$, the latter ring is an Ore domain (see \cref{subsect:localization}) and since the map $\DC_R\ast \ZZ \hookrightarrow \Ore(\DC_R\ast \ZZ)$ is also epic and $\Ore(\DC_R\ast \ZZ)$ is a division ring, the composition $R\ast \ZZ \hookrightarrow \Ore(\DC_R\ast \ZZ)$ is a division $R\ast \ZZ$-ring of fractions. This finishes the proof of $(i)$.\medskip
   
   Finally, note that since $R\ast \ZZ$ is isomorphic to $R[t^{\pm 1};\tau]$ as a ring, $(ii)$ follows from the fact that the left $R[t^{\pm 1};\tau]$-linear map
    \begin{align*}
        R[t^{\pm 1};\tau]\otimes_R \DC_R&\to \DC_R[t^{\pm 1};\tilde \tau] \\
        t^n\otimes \lambda &\mapsto \tilde \tau^n(\lambda)t^n
    \end{align*}
    is an isomorphism since it is also right $\DC_R$-linear and maps the basis $\{t^n \otimes 1\mid n\in\ZZ\}$ to the basis $\{t^n\mid n\in\ZZ\}$. The statement for right modules is proved analogously.
\end{proof}

We are interested in the homological properties of $\DC_{\SC} = \Ore(\DC_{\F}\ast \ZZ)$, to which we will dedicate \cref{subsect:homology_DC_SC}. In the previous lemma we explored the $\SC$-module structure of $\DC_{\F}\ast \ZZ$, while the next one, applied to the case $R\coloneqq\DC_{\F}\ast\ZZ$, $\mathcal O\coloneqq\DC_{\SC}$ and $S \coloneqq \SC$, will allow us later to restrict our attention to $\SC$-submodules of $\DC_{\F}\ast\ZZ$.
\begin{lemm}
    \label{lemm:submodules_ore}
    Let $R$ be a right Ore domain with right Ore localization $\mathcal O$ and $S$ a subring of $R$.
    Then every finitely generated $S$-submodule $M$ of the left $S$-module $\mathcal O$ is isomorphic to a finitely generated $S$-submodule of $R$.
\end{lemm}
\begin{proof}
    Let $M$ be generated as a left $S$-module by $x_1,\ldots,x_m\in \mathcal O$.
    We find $p_{i}, q_{i}\in R$ such that $x_{i}=p_{i} q_{i}^{-1}$ for $i=1,\ldots,m$. If $m\ge 2$ we can use the Ore condition to find non-zero $a, b\in R$ such that $q_{1}a=q_{2}b$, and hence $x_1=(p_{1}a)(q_{1}a)^{-1}$ and $x_2=(p_{2}b)(q_{2}b)^{-1}$ can be expressed as fractions with common denominators.
    By repeatedly applying this procedure we produce $p_{i}', q\in R, q\ne 0$ such that $x_{i}=p'_{i}q^{-1}$ for all $i$.

    We now consider the left $S$-submodule $M'$ of $R$ generated by $x_1q, \ldots,x_mq$.
    The map $f\colon M\to M'$ given by $y\mapsto yq$ is $S$-linear since $\mathcal O$ is associative and surjective since its image contains the generators.
    Finally, it is injective, since $\mathcal O$ is a division ring and hence $zq\ne0$ for every $z \ne 0$.
    We conclude that $f$ is an $S$-linear isomorphism.
\end{proof}

\subsection{Homological recognition principles for (pseudo-)Sylvester domains} \label{subsect:criteria}

As mentioned above, we are going to use the next two theorems to prove \cref{theo:main_fir_crossed_ZZ}. The first one is a direct consequence of the new characterization of Sylvester domains \cite{Jaikin2019_sylvester}*{Proposition~2.2 \& Theorem 2.4} provided by Jaikin-Zapirain.

\begin{theo}
   \label{theo:sylvester_criterion}
    Let $R\hookrightarrow \DC$ be a division $R$-ring of fractions.
    Assume that
    \begin{enumerate}
        \item $\Tor_1^R(\DC, \DC) = 0$ and
        \item for any finitely generated left or right $R$-submodule $M$ of $\DC$ and any exact sequence $0\to J\to R^n\to M\to 0$, the $R$-module $J$ is free of finite rank.
    \end{enumerate}
    Then $R$ is a Sylvester domain and $\DC$ is the universal division $R$-ring of fractions.
\end{theo}

The second theorem is an analogue for pseudo-Sylvester domains, involving stably free modules instead of free modules.
The proof proceeds similarly, but we include it here for the sake of completeness.
Given an embedding $R\hookrightarrow \DC$ of $R$ into a division ring and a matrix $A$ over $R$, we will denote by $\rk_{\DC}(A)$ the usual $\DC$-rank of $A$ considered as a matrix over $\DC$.
Similarly, if $M$ is a left $R$-module, we take $\dim_{\DC}(M)$ to denote the $\DC$-dimension $\dim_{\DC}(\DC\otimes_R M)$ of the left $\DC$-module $\DC \otimes_R M$.

\begin{theo}
   \label{theo:pseudosylvester_criterion}
    Let $R\hookrightarrow \DC$ be a division $R$-ring of fractions.
    Assume that
    \begin{enumerate}
        \item $\Tor_1^R(\DC, \DC) = 0$ and
        \item for any finitely generated left or right $R$-submodule $M$ of $\DC$ and any exact sequence $0\to J\to R^n\to M\to 0$, the $R$-module $J$ is finitely generated stably free.
    \end{enumerate}
    Then $R$ is a pseudo-Sylvester domain and $\DC$ is the universal division $R$-ring of fractions.
\end{theo}

\begin{proof}
 Notice that by \cref{prop:equivalence} it suffices to show that every stably full matrix over $R$ becomes invertible over $\DC$. Thus, let $A$ be an $n\times n$ matrix over $R$ with $\rho^{\ast}(A) = n$, and assume that $A$ is not invertible over $\DC$, i.e., $\rk_{\DC}(A)<n$.
 Since $R$ is a subring of a division ring, it is necessarily stably finite.

 Let $N$ be the left $R$-module $N = R^{n}/R^nA$. Then $A$ is also the presentation matrix of $\DC\otimes_R N$, and therefore $\dim_{\DC}(N) = n-\rk_{\DC}(A)$, which is finite and positive. This implies that $\DC \otimes_R N \cong \DC^k$ as $\DC$-modules for some $k\ge 1$ and, thus, composing the $R$-homomorphism $N\to \DC\otimes_R N$ given by $x\to 1\otimes x$ with an appropriate projection, we obtain a non-trivial $R$-homomorphism $N\to \DC$. Therefore, if $M$ is the image of this map, the surjection $N\to M$ gives us a commutative diagram with exact rows:
\[
   \begin{tikzcd}
    0 \arrow[r] & R^nA \arrow[r] \arrow[d, dotted] & R^n \arrow[r] \arrow[d, equal] & N \arrow[r] \arrow[d] & 0   \\
    0 \arrow[r] & J \arrow[r] & R^n \arrow[r] & M \arrow[r] & 0.
   \end{tikzcd}
\]
Here, $J$ is the kernel of the map $R^n\to M$ and the dotted arrow is such that the left square commutes (cf.\ \cite{Rotman2009}*{Proposition 2.71}) and therefore injective.
Moreover, notice that $\DC\otimes_R M$ is non-trivial since the multiplication map to $\DC$ is non-trivial.
We conclude that $\dim_{\DC}(M)>0$.

Now we have by $(2)$ that $J$ is stably free, i.e., there exists $s\geq 0$ such that $J\oplus R^s$ is free.
Moreover, since $J$ is finitely generated and $R$, as a subring of a division ring, is stably finite, we conclude that $J\oplus R^s\cong R^{\rk_{sf}(J)+s}$.
In fact, we obtain that $\rk_{sf}(J) = \dim_{\DC}(J)$ by applying $\DC \otimes_R \square$.
Notice also that the previous diagram remains exact and commutative if we add $0\to R^s \to R^s\to 0 \to 0$ to both rows. Thus, setting $t\coloneqq \dim_{\DC}(J)$, the situation can be summarized in the following commutative diagram:
\[
\begin{tikzcd}[column sep=large, row sep=large, ar symbol/.style = {draw=none,"\textstyle#1" description,sloped}, isomorphic/.style = {ar symbol={\cong}}]
   &[-37pt] R^{n+s} \arrow[d, "r_{A\oplus I_s}"', two heads] \arrow[dr, "r_{A\oplus I_s}"]\\
   & R^n A\arrow[d, hook]\oplus R^s \arrow[r, hook] & R^{n+s} \arrow[d, equal] \\
   R^{t+s}\arrow[r, isomorphic] & J\oplus R^s\arrow[r, hook] & R^{n+s}.
\end{tikzcd}
\]

Here, $r_{A\oplus I_s}$ denotes the homomorphism given by right multiplication by $A\oplus I_s$, so that all maps except the isomorphism behave identically on the $R^s$ summand.
In terms of matrices, this factorization of $r_{A\oplus I_s}$ allows us to express $A\oplus I_s$ as a product of two matrices of dimensions $(n+s)\times (t+s)$ and $(t+s)\times (n+s)$, respectively.
Thus, $\rho(A\oplus I_s)\leq t+s$ right by definition.
Since $A$ is stably full, we have $\rho(A\oplus I_s) = n+s$ for every $s$, so we conclude that $n\le t$.

We are going to show on the other hand that $t<n$, a contradiction. Observe first that the condition (2) tells us in particular that the flat (in fact, projective) dimension of any finitely generated right $R$-submodule of $\DC$ is at most 1.
Hence, using \cref{lemm:tor_rotman} and the fact that $\Tor$ commutes with directed colimits (cf.\ \cite{Rotman2009}*{Proposition~7.8}), we obtain that for any left $R$-module $Q$,
\[\Tor_{2}^R(\DC, Q) = \Tor_{2}^R\left( \lim_{\longrightarrow} L_i,Q\right) \cong \lim_{\longrightarrow} \Tor_2^R (L_i,Q) = 0,\]
where $L_i$ runs through all finitely generated $R$-submodules of the right $R$-module $\DC$. Again by \cref{lemm:tor_rotman}, this means that $\DC$ itself has flat dimension at most 1 as a right $R$-module.

Now, since $M$ is an $R$-submodule of $\DC$, we have an exact sequence of left $R$-modules $0\to M \to \DC \to Q\to 0$ for some left $R$-module $Q$, and hence, applying $\DC\otimes_R \square$ we can construct a long exact sequence containing the following exact part:
\[ \cdots \to \Tor_2^R(\DC, Q)\to \Tor_1^R(\DC, M) \to \Tor_1^R(\DC,\DC) \to \cdots.\]
The first term is trivial by the previous argument, while the third term is trivial because of (1).
Thus, we deduce that $\Tor_1^R(\DC, M) = 0$. From here, it follows that applying $\DC \otimes_R \square$ to the exact sequence $0\to J \to R^n\to M \to 0$ returns an exact sequence of left $\DC$-modules
\[0\to \DC\otimes_R J \to \DC^n\to \DC\otimes_R M \to 0,\]
from which we obtain
\[
\begin{matrix*}[l]
  t = \dim_{\DC}(J) = n - \dim_{\DC}(M) < n.
\end{matrix*}
 \]
This is the desired contradiction, which shows that necessarily $\rk_{\DC}(A) = n$.
\end{proof}

In the case of $\F\ast \ZZ$, the role of $\DC$ will be played by the Ore division ring of fractions $\DC_{\SC} = \Ore(\DC_{\F}\ast \ZZ)$.

\subsection{The homological properties of \texorpdfstring{$\DC_{\SC}$}{D\_S} and the proof of Theorem A}
\label{subsect:homology_DC_SC}

We will now study the homological properties of the $\SC$-module $\DC_{\SC}$ and its submodules.
In particular, we will derive vanishing results for $\Tor$ and $\Ext$, which will allow us to verify condition (1) and a weak version of condition (2) of \cref{theo:pseudosylvester_criterion,theo:sylvester_criterion}. From this, we will finally derive \cref{theo:main_fir_crossed_ZZ}.

The following theorem, which combines Theorem~4.7 and 4.8 of \cite{Schofield1985}, will be very useful in verifying condition (1):
\begin{theo}
    \label{theo:schofield}
    Let $R\to S$ be an epic ring homomorphism. Then the following are equivalent:
    \begin{enumerate}
        \item $\Tor_1^R(S, S)=0$.
        \item $\Tor_1^R(M, N)=\Tor_1^{S}(M, N)$ for every right $S$-module $M$ and every left $S$-module $N$.
        \item $\Ext_R^1(M, M')=\Ext_{S}^1(M, M')$ for all right $S$-modules $M$ and $M'$.
        \item $\Ext_R^1(N, N')=\Ext_{S}^1(N, N')$ for all left $S$-modules $N$ and $N'$.
    \end{enumerate}
    If $S=R_\Sigma$ is a universal localization of $R$, then all of these properties are satisfied.
\end{theo}

The importance of this theorem in our paper is given by the fact that, since firs are Sylvester domains, the universal division $\F$-ring of fractions $\DC_{\F}$ is precisely the universal localization of $\F$ with respect to the set of all full matrices. Therefore, each of the statements in \cref{theo:schofield} holds for the epic embedding $\F\hookrightarrow \DC_{\F}$, and this will serve as the starting point for the proof of the main result. The other crucial property in our setting is the following.

\begin{lemm} \label{lemm:global_dimension2}
    Let $R$ be a ring of right (resp.\ left) global dimension at most 1.
    Then any crossed product $R\ast\ZZ$ has right (resp.\ left) global dimension at most 2.
    In particular, if $\F$ is a fir, then $\F\ast\ZZ$ has right and left global dimension at most 2.
\end{lemm}
\begin{proof}
    This can be found in \cite{MR2001}*{7.5.6~Corollary~(ii)} noting the symmetry in the definition of a crossed product. The last statement follows because firs have right and left global dimension at most 1.
\end{proof}

We are now ready to study the homological properties of $\DC_{\SC}$ and its submodules.

\begin{lemm}
    \label{lemm:ext_vanishing2} \leavevmode
    \begin{enumerate}[label={(\arabic*)},ref={\thetheo~(\arabic*)}]
        \item $\Ext_{\SC}^3(M, M')=0$ for all left (resp.\ right) $\SC$-modules $M$ and $M'$.
        \item \label[lemmenum]{lemm:ext_vanishing:dc_f_z_2} $\DC_{\F}\ast\ZZ$ has projective dimension at most 1 as a left and right $\SC$-module.

        \item \label[lemmenum]{lemm:ext_vanishing:sub_dc_f_z_2} Every left or right $\SC$-submodule of $\DC_{\F}\ast \ZZ$ has projective dimension at most 1.

        \item \label[lemmenum]{lemm:ext_vanishing:projective_dimension_1} Every finitely generated left or right $\SC$-submodule of $\DC_{\SC}$ has projective dimension at most 1.
    \end{enumerate}
\end{lemm}
\begin{proof}
    \noindent (1) Since $\SC$ has global dimension at most 2 by \cref{lemm:global_dimension2}, this is a consequence of \cref{lemm:ext_rotman}.
    \smallskip

    \noindent (2) Since $\F$ has global dimension at most 1, the left $\F$-module $\DC_{\F}$ admits a resolution $0\to P_1\to P_0\to \DC_{\F}\to 0$ with $P_1$ and $P_0$ projective left $\F$-modules.
    We now apply the functor $\SC\otimes_{\F}\square$ to this short exact sequence, where we view $\SC$ as an $\SC$-$\F$-bimodule.
    Since $\SC$ is a free right $\F$-module, the resulting sequence is a projective resolution of the left $\SC$-module $\SC\otimes_{\F} \DC_{\F}$, and thus the projective dimension of this module is at most 1.
    This finishes the proof, since the left $\SC$-modules $\SC\otimes_{\F}\DC_{\F}$ and $\DC_{\F}\ast\ZZ$ are isomorphic by \cref{lemm:tensor_vs_crossed_product2}. The corresponding statement for the right $\SC$-module $\DC_{\F}\ast\ZZ$ follows analogously.\smallskip

    \noindent (3) For every left (resp.\ right) $\SC$-module $M'$, the Ext long exact sequence obtained by applying the functor $\Hom_{\SC}(\square, M')$ to the short exact sequence $0\to M\to \DC_{\F}\ast \ZZ\to Q\to 0$ for an appropriate $\SC$-module $Q$ contains the following exact part:
    \[\ldots\to \Ext_{\SC}^2(\DC_{\F}\ast\ZZ, M')\to \Ext_{\SC}^2(M, M')\to \Ext_{\SC}^3(Q, M')\to\ldots\]
    Here, the first term vanishes by (2) and \cref{lemm:ext_rotman}, and the third term vanishes by property (1).
    By exactness, we conclude that the term in the middle also vanishes. Thus, the claim follows from \cref{lemm:ext_rotman}.
    \smallskip

    \noindent (4) This follows directly from (3) and \cref{lemm:submodules_ore}.
\end{proof}

\begin{lemm}
    \label{lemm:tor_vanishing2}
    \leavevmode
    \begin{enumerate}[label={(\arabic*)},ref={\thetheo~(\arabic*)}]
        \item $\Tor_1^{\F}(\DC_{\F}, \DC_{\F})=0$.
        \item $\Tor_2^{\SC}(\DC_{\F} \ast \ZZ, N)=0$ for every left $\SC$-module $N$.
        \item $\Tor_1^{\SC}(\DC_{\F} \ast \ZZ, N)=0$ for every left $\DC_{\F}\ast\ZZ$-module $N$.
        \item $\Tor_1^{\SC}(\DC_{\F} \ast \ZZ, N)=0$ for every left $\SC$-submodule $N\le \DC_{\SC}$.
        \item \label[lemmenum]{lemm:tor_vanishing:dc_sc_sub_1} $\Tor_1^{\SC}(\DC_{\SC}, N)=0$ for every left $\SC$-submodule $N\le \DC_{\SC}$.
        \item $\Tor_1^{\SC}(N, \DC_{\SC})=0$ for every right $\SC$-submodule $N\le \DC_{\SC}$.
        \item \label[lemmenum]{lemm:tor_vanishing:dc_sc_1} $\Tor_1^{\SC}(\DC_{\SC}, \DC_{\SC})=0$.
    \end{enumerate}
\end{lemm}
\begin{proof}
    (1) Since $\F$ is a fir, we know that $\DC_{\F}$ is the universal localization of $\F$ with respect to the set of all full matrices, so this follows from \cref{theo:schofield}.
    \smallskip

    \noindent (2) The flat dimension of a module is at most its projective dimension, so this follows from \cref{lemm:ext_vanishing:dc_f_z_2} and \cref{lemm:tor_rotman}.\smallskip

    \noindent (3) Observe that $\DC_{\F}\ast \ZZ$ is isomorphic to $\DC_{\F} \otimes_{\F} \SC$ as a right $\SC$-module by \cref{lemm:tensor_vs_crossed_product2} and that $\SC$ is a free left $\F$-module (in particular flat). Thus, \cref{lemm:shapiro}, together with $(1)$ and \cref{theo:schofield}~(2), tells us that
    \[
    \Tor_1^{\SC}(\DC_{\F} \ast \ZZ, N) \cong \Tor_1^{\F}(\DC_{\F}, N) \cong \Tor_1^{\DC_{\F}}(\DC_{\F}, N) = 0.
    \]

    \noindent (4) We have a short exact sequence $0\to N\to \DC_{\SC}\to Q \to 0$ for some left $\SC$-module $Q$. Applying $\DC_{\F}\ast \ZZ \otimes_{\SC} \square$ to this sequence, we obtain a long exact sequence that contains the following subsequence:
    \[
       \ldots\to\Tor_2^{\SC}(\DC_{\F} \ast \ZZ, Q)\to \Tor_1^{\SC}(\DC_{\F} \ast \ZZ, N)\to\Tor_1^{\SC}(\DC_{\F} \ast \ZZ, \DC_{\SC})\to\ldots
    \]
    Since the first and third term vanish by (2) and (3), respectively, we obtain the result.\smallskip

    \noindent (5)
    Let
    \[
      \ldots\to P_k\to\ldots\to P_0\to N\to 0
    \]
    be a projective resolution of $N$.
    We can compute $\Tor_1^{\SC}(\DC_{\SC}, N)$ as the first homology group of the $\SC$-chain complex
    \[
     \ldots\to \DC_{\SC} \otimes_{\SC} P_k \to\ldots\to \DC_{\SC} \otimes_{\SC} P_0\to 0.
    \]
    Since $\DC_{\SC}\otimes_{\SC} \square \cong \DC_{\SC}\otimes_{\DC_{\F}\ast \ZZ}\DC_{\F}\ast\ZZ\otimes_{\SC} \square$, this complex is $\SC$-isomorphic to:
    \[
        C_*\colon\quad\ldots\to \DC_{\SC} \otimes_{\DC_{\F}\ast\ZZ} \DC_{\F}\ast\ZZ\otimes_{\SC} P_k \to\ldots\to \DC_{\SC} \otimes_{\DC_{\F}\ast\ZZ} \DC_{\F}\ast\ZZ \otimes_{\SC} P_0\to 0.
    \]
    Using that $\DC_{\SC}$ is the Ore localization of $\DC_{\F}\ast\ZZ$, which implies that the functor $\DC_{\SC}\otimes_{\DC_{\F}\ast \ZZ} \square$ is exact, we obtain that $H_*(C_*)\cong\DC_{\SC}\otimes_{\DC_{\F}\ast \ZZ} H_*(D_*)$, where
    \[
        D_*\colon\quad\ldots\to \DC_{\F}\ast\ZZ\otimes_{\SC} P_k \to\ldots\to \DC_{\F}\ast\ZZ \otimes_{\SC} P_0\to 0.
    \]
    But the homology of this complex computes $\Tor_k^{\SC}(\DC_{\F}\ast \ZZ, N)$, and thus
    \[
        \Tor_1^{\SC}(\DC_{\SC}, N)\cong H_1(C_*)\cong \DC_{\SC}\otimes_{\DC_{\F}\ast \ZZ} H_1(D_*) \cong \DC_{\SC}\otimes_{\DC_{\F}\ast \ZZ} \Tor_1^{\SC}(\DC_{\F}\ast \ZZ, N) \stackrel{(4)}{=} 0.
    \]

    \noindent(6)
    Every step in the proof of (5) can be adapted for right modules since $\SC$ is also a free right $\F$-module, and we can apply \cref{lemm:ext_vanishing2}, \cref{lemm:tensor_vs_crossed_product2} and the corresponding version of \cref{lemm:shapiro} for right modules. \smallskip

    \noindent (7)
    This is a special case of (5).
\end{proof}

We obtain from the previous results a weaker version of conditions (2) of \cref{theo:sylvester_criterion} and \cref{theo:pseudosylvester_criterion}:
\begin{prop}
    \label{prop:j_finitely_generated_projective}
    For every finitely generated left or right $\SC$-submodule $M$ of $\DC_{\SC}$ and every exact sequence $0\to J\to \SC^n\to M\to 0$, the $\SC$-module $J$ is finitely generated projective.
\end{prop}
\begin{proof}
    Since $M$ has projective dimension at most 1 by \cref{lemm:ext_vanishing:projective_dimension_1} and $\SC^n$ is projective, it follows from Schanuel's lemma that $J$ is projective.

    If $M$ is a left $\SC$-module and we apply the functor $\DC_{\SC}\otimes_{\SC} \square$ to the short exact sequence defining $J$, the sequence remains exact by \cref{lemm:tor_vanishing:dc_sc_sub_1}.
    In particular, $\DC_{\SC}\otimes_{\SC} J$ is isomorphic to a $\DC_{\SC}$-submodule of the finitely generated $\DC_{\SC}$-module $(\DC_{\SC})^n$.
    But $\DC_{\SC}$ is a division ring, thus $\DC_{\SC}\otimes_{\SC} J$ is itself finitely generated.
    Since $J$ is projective, \cite{LLS2003}*{Lemma~4} applies and we obtain that $J$ is finitely generated.
\end{proof}

We finally have all the necessary ingredients for the proof of \cref{theo:main_fir_crossed_ZZ}.

\begin{proof}[Proof of \cref{theo:main_fir_crossed_ZZ}]
    By \cref{lemm:tor_vanishing:dc_sc_1}, the conditions (1) of \cref{theo:pseudosylvester_criterion} and \cref{theo:sylvester_criterion} are satisfied for $\SC \hookrightarrow \DC_{\SC}$, while we obtain from \cref{prop:j_finitely_generated_projective} that the module $J$ appearing in the conditions (2) is finitely generated and projective. Therefore, if every finitely generated projective $\SC$-module is stably free (resp.\ free), we deduce that $\SC$ is a pseudo-Sylvester domain (resp.\ Sylvester domain). Conversely, over a pseudo-Sylvester domain every finitely generated projective module is stably free (cf.\ \cite{Cohn2006}*{Proposition 5.6.2}), while Sylvester domains are always projective-free (cf.\ \cite{Cohn2006}*{Proposition 5.5.7}).

    In any of the previous cases, we conclude from the criteria that $\DC_{\SC} = \Ore(\DC_{\F}\ast \ZZ)$ is the universal division $\F\ast \ZZ$-ring of fractions, and hence isomorphic to the universal localization of $\F\ast \ZZ$ with respect to the set of all stably full (resp.\ full) matrices.
\end{proof}

As we mentioned in \cref{subsect:Sylvester}, one could also use the results of Cohn and Schofield in \cite{CS1982} to deduce \cref{theo:main_fir_crossed_ZZ} b) from a).

\section{Proof of \texorpdfstring{\cref{theo:main_pseudosylvester}}{Theorem B} and examples}\label{sect:applications}

The aim of this section is to prove \cref{theo:main_pseudosylvester}. Thus, throughout this section the main object of study will be a crossed product $E \ast G$, where $E$ is a division ring and $G$ denotes a group that fits into a short exact sequence
\[1\to F\to G\to \ZZ \to 1,\] with $F$ a non-necessarily finitely generated free group. Since $\ZZ$ is a free group, any such extension splits and $G$ arises as a semi-direct product $F\rtimes \ZZ$.

As in \cref{subsect:crossed_products}, in this case $E\ast G$ can be expressed as an iterated crossed product $(E\ast F)\ast \ZZ$, and since $E\ast F$ is a fir, we are in the situation of \cref{theo:main_fir_crossed_ZZ} with $\F=E\ast F$ and $\SC=E\ast G$. In this section, we use $\DC_{E\ast F}$ to denote the universal division $E\ast F$-ring of fractions and set $\DC_{E\ast G} = \Ore(\DC_{E\ast F}\ast \ZZ)$.

In \cref{subsect:Farrell-Jones} we use the Farrell--Jones conjecture in algebraic $K$-theory to show that $E\ast G$ is always a pseudo-Sylvester domain. Whether this ring is even a Sylvester domain is a much more delicate question and not much can be said in general. In \cref{subsect:examples} we give examples of group rings for which this question has a known answer.

\subsection{The Farrell--Jones conjecture and stably freeness}\label{subsect:Farrell-Jones}

In this subsection we use recent results on the Farrell--Jones conjecture to prove that the finitely generated projective $E\ast G$-module $J$ that appears in condition (2) of \cref{theo:pseudosylvester_criterion} is actually stably free, which will conclude the first part of the proof of \cref{theo:main_pseudosylvester}.
The following piece of the algebraic $K$-theory of a ring is needed to phrase the results:
\begin{defi}
    Let $R$ be a ring.
    Then we denote by $K_0(R)$ the abelian group generated by the isomorphism classes $[P]$ of finitely generated projective $R$-modules together with the relations
    \[ [P \oplus Q]-[P]-[Q]=0\]
    for all finitely generated projective $R$-modules $P$ and $Q$.
\end{defi}

Every element of $K_0(R)$ is of the form $[P]-[P']$ for finitely generated projective $R$-modules $P$ and $P'$.
The identity $[P]=[P']\in K_0(R)$ holds for two finitely generated projective $R$-modules $P$ and $P'$ if and only if there is a finitely generated projective $R$-module $Q$ such that $P\oplus Q\cong P'\oplus Q$, where $Q$ can even be taken to be free.

If $f\colon R\to S$ is a ring homomorphism and $P$ is a finitely generated projective $R$-module, then $S\otimes_R P$ is a finitely generated projective $S$-module.
In this way, $K_0(\square)$ becomes a functor from rings to abelian groups.

The conditions of \cref{rema:left_right} are satisfied for $K_0(\square)$ and thus it does not depend on whether we use left or right modules in its definition.

The \emph{Farrell--Jones conjecture} makes far-reaching claims about the $K$-theory and $L$-theory of group rings or, more generally, additive categories with group actions, in particular for torsion-free groups.
It is known for many classes of groups and satisfies a number of useful inheritance properties.
For a full statement of the Farrell--Jones conjecture and an overview of the groups for which it is known, we refer the reader to the surveys \cite{BLR2008} and \cite{RV2018}, and also to the book project \cite{Luck2019}.

We will need the following consequence of the Farrell--Jones conjecture which is certainly well-known, but has not been made explicit in the literature.
\begin{prop}
    \label{theo:fjca_k0}
    Let $E$ be a division ring, $\Gamma$ a torsion-free group and $E\ast \Gamma$ a crossed product.
    If the $K$-theoretic Farrell--Jones conjecture with coefficients in an additive category holds for $\Gamma$, then the embedding $E\hookrightarrow E\ast \Gamma$ induces an isomorphism
    \[K_0(E)\xrightarrow{\cong} K_0(E\ast \Gamma).\]
    In particular, since $K_0(E)=\{n[E]\mid n\in\ZZ\}$, every finitely generated projective $E\ast \Gamma$-module is stably free.
\end{prop}

\begin{proof}
    For a given crossed product $E\ast \Gamma$, we will denote the additive category defined in \cite{BR2007}*{Corollary~6.17} by $\AC_{E\ast \Gamma}$.
    We will freely use the terminology and notation of that paper.
    Furthermore, we will denote the family of virtually cyclic subgroups of a given group by $\VCyc$ and the family consisting just of the trivial subgroup by $\Triv$.
    The $K$-theoretic Farrell--Jones conjecture for the group $\Gamma$ with coefficients in the additive category $\AC_{E\ast \Gamma}$ arises as an instance of the more general meta-isomorphism conjecture \cite{Luck2019}*{Conjecture~13.2} for the $\Gamma$-homology theory $\HC^\Gamma_*(\square; \KB_{\AC_{E\ast \Gamma}})$ introduced in \cite{BR2007} and the family $\FC=\VCyc$.
    It states that the assembly map
    \[\HC^\Gamma_*(\EC_{\VCyc}(\Gamma); \KB_{\AC_{E\ast \Gamma}})\to \HC^\Gamma_*(\pt; \KB_{\AC_{E\ast \Gamma}})\]
    is an isomorphism, where the right-hand side is isomorphic to $K_*(E\ast \Gamma)$ by \cite{BR2007}*{Corollary~6.17}.

    In order to arrive at the desired conclusion, we need to reduce the family from $\VCyc$ to $\Triv$.
    Since $\Gamma$ is assumed to be torsion-free and hence all its virtually cyclic subgroups are infinite cyclic, we can arrange for this via the transitivity principle of~\cite{Luck2019}*{Theorem~13.13 (i)} if the meta-isomorphism conjecture holds for the $\ZZ$-homology theory $\HC^\ZZ_*(\square; \KB_{\AC_{E\ast \ZZ}})$ and the family $\FC=\Triv$.
    A model for the classifying space $\EC_{\Triv}(\ZZ)$ is given by $\RR$ and we may again assume that the crossed product $E\ast\ZZ$ is a skew Laurent polynomial ring $E[t^{\pm 1};\tau]$.
    In this situation, since $E$ is regular, i.e. it is Noetherian and every finitely generated $E$-module possesses a finite resolution by finitely generated projective modules, the assembly map coincides with the map provided by the analogue of the Fundamental Theorem of algebraic $K$-theory for skew Laurent polynomial rings, which is an isomorphism (cf.~\cite{BL2020}*{Theorems~6.8~\&~9.1} or~\cite{Grayson1985} for a more classical treatment).

    Since the $K$-theoretic Farrell--Jones conjecture with coefficients in an additive category is assumed to hold for $\Gamma$, we now obtain from the transitivity principle that the assembly map
    \[\HC^\Gamma_*(\EC_{\Triv}(\Gamma); \KB_{\AC_{E\ast \Gamma}})\to \HC^\Gamma_*(\pt; \KB_{\AC_{E\ast \Gamma}}) \cong K_*(E\ast \Gamma)\]
    is an isomorphism.
    The space $\EC_{\Triv}$ is a free $\Gamma$-space and the value at the coset $\nicefrac{\Gamma}{\{1\}}$ of the $\Or(\Gamma)$-spectrum $\KB_{\AC_{E\ast \Gamma}}$ is $\KBB(\AC_{E\ast \Gamma}\ast \nicefrac{\Gamma}{\{1\}})$.
    We can thus simplify the left-hand side of the assembly map as follows:
    \[\HC^\Gamma_*(\EC_{\Triv}(\Gamma); \KB_{\AC_{E\ast \Gamma}})\cong H_*(B\Gamma; \KBB(\AC_{E\ast \Gamma}\ast \nicefrac{\Gamma}{\{1\}})).\]
    Here, $B\Gamma$ denotes the standard classifying space of the group $\Gamma$ and homology is taken with local coefficients.
    Using \cite{BR2007}*{Corollary~6.17} once more, we observe that $\KBB(\AC_{E\ast \Gamma}\ast \nicefrac{\Gamma}{\{1\}})$ is weakly equivalent to $\KBB(E)$, which is connective by \cite{Luck2019}*{Theorem~3.6} since $E$ is a regular ring.
    In particular, the Atiyah--Hirzebruch spectral sequence provides the following natural isomorphism:
    \[H_0(B\Gamma; \KBB(\AC_{E\ast \Gamma}\ast \nicefrac{\Gamma}{\{1\}}))\cong H_0(B\Gamma; \pi_0(\KBB(\AC_{E\ast \Gamma}\ast \nicefrac{\Gamma}{\{1\}}))),\]
    where homology is again taken with local coefficients.
    Since $\pi_0(\KBB(\AC_{E\ast \Gamma}\ast \nicefrac{\Gamma}{\{1\}}))\cong K_0(\AC_{E\ast \Gamma}\ast \nicefrac{\Gamma}{\{1\}})$ and the $\Gamma$-action on $\AC_{E\ast \Gamma}\ast \nicefrac{\Gamma}{\{1\}}$, which is induced from that on the $\Gamma$-space $\nicefrac{\Gamma}{\{1\}}$, preserves isomorphism types, the local coefficients are in fact constant.
    We conclude that
    \[H_0(B\Gamma; \KBB(\AC_{E\ast \Gamma}\ast \nicefrac{\Gamma}{\{1\}}))\cong H_0(B\Gamma; K_0(E)),\]
    and thus the assembly map in degree 0 simplifies to
    \[K_0(E)\cong H_0(B\Gamma; K_0(E))\xrightarrow{\cong} K_0(E\ast \Gamma).\]
    This proves the first statement.

    The second statement is a direct consequence since every finitely generated projective $E\ast \Gamma$-module $P$ represents an element $n[E\ast \Gamma]$ in $K_0(E\ast \Gamma)$ for $n\geq 0$ and thus there exists a finitely generated free $E\ast \Gamma$-module $Q$ such that $P\oplus Q\cong (E\ast \Gamma)^n\oplus Q$, which is free.
\end{proof}

The following is the $K$-theoretic part of~\cite{BFW2019}*{Theorem~1.1} in the case of a finitely generated free group $F$ and~\cite{BKW2019}*{Theorem~A} in the general case:
\begin{theo}
    \label{theo:fjc_free_by_cyclic}
    The $K$-theoretic Farrell--Jones conjecture with coefficients in an additive category holds for every group that arises as an extension
    \[1\to F\to G\to \ZZ\to 1\]
    with $F$ a (not necessarily finitely generated) free group.
\end{theo}

The previous result provides the final step in the proof of \cref{theo:main_pseudosylvester}.

\begin{proof}[Proof of \cref{theo:main_pseudosylvester}]
    Since $G$ satisfies the $K$-theoretic Farrell--Jones conjecture with coefficients in additive categories by \cref{theo:fjc_free_by_cyclic}, we obtain from \cref{theo:fjca_k0} that every finitely generated projective $E\ast G$-module is stably free.
    Therefore, the statement follows from \cref{theo:main_fir_crossed_ZZ}.
\end{proof}

\subsection{Examples and non-examples of Sylvester domains} \label{subsect:examples}

The main examples of groups of the form $1\to F\to G\to \ZZ\to 1$ are the free-by-\{infinite cyclic\} groups (terminology usually reserved in the literature for the case where $F$ is finitely generated) and fundamental groups of closed surfaces with genus $g\ge 1$ other than the projective plane, which has to be excluded since its fundamental group has torsion. In the latter family, we have to distinguish the fundamental groups $S_g$ of orientable closed surfaces of genus $g\ge 1$, which admit the presentations
\[S_g = \langle a_1, b_1,\ldots, a_g,b_g\mid [a_1,b_1]\cdot\ldots\cdot [a_g,b_g]\rangle,\]
and the fundamental groups of non-orientable closed surfaces of genus $g\ge 2$,
which admit the presentations
\[\mathfrak S_g = \langle a_1,\ldots, a_g \mid a_1^2\cdot\ldots\cdot a_g^2\rangle.\]
That these groups contain a normal free subgroup $F$ such that $G/F$ is infinite cyclic is a consequence of the fact that their infinite index subgroups are free (cf.\ \cite{HKS1972}) and that their abelianizations contain an infinite cyclic summand.

Within these families, there are some cases of group rings for which it is known whether they admit stably free cancellation. In the following examples, $K$ is any field of characteristic 0.

\begin{itemize}
    \item Examples of group rings with stably free cancellation are $K[\ZZ^2] = K[S_1]$ (cf.\ \cite{Swan1978}) and $K[F_2\times\ZZ]$ (cf.\ \cite{Bass1968}*{IV.6.4}, using that $K[\ZZ]$ is a PID and thus a projective-free Dedekind domain).
    \item Examples of group rings which do admit non-free stably free modules are given by $K[\ZZ\rtimes\ZZ]=K[\mathfrak S_2]$ (cf.\ \cite{Stafford1985}*{Theorem~2.12}) and $\QQ[\langle x,y\mid x^3=y^2\rangle]=\QQ[F_2\rtimes\ZZ]$ (cf.\ \cite{Lewin1982} and note that the non-free projective ideal in the main theorem is actually stably free). Here, the latter example is the rational group ring of the fundamental group of the complement of the trefoil knot, which fibers over the circle and hence admits a free-by-\{infinite cyclic\} fundamental group (cf.\ \cite{BZH2013}*{Corollary~4.12}).
    Both group rings serve as examples of pseudo-Sylvester domains that are not Sylvester domains.
\end{itemize}

To the best of the authors' knowledge, it is an open question whether $\CC[S_g]$ for $g\ge 2$ and $\CC[\mathfrak S_g]$ for $g\ge 3$ have stably free cancellation.

\section{Identifying \texorpdfstring{$\DC_{E\ast G}$}{DE*G} via Hughes-freeness} \label{subsect:locally_indicable}

We will now identify the universal division $E\ast G$-ring of fractions $\DC_{E\ast G}$ of \cref{theo:main_pseudosylvester} with division rings constructed by Gräter and Linnell. For this purpose, we introduce the notion of Hughes-free division $E\ast \Gamma$-ring of fractions for a locally indicable group $\Gamma$ and make use of the fact that, whenever it exists, it is unique up to $E\ast \Gamma$-isomorphism (\cite{Hughes1970}).

Recall that the \emph{division closure} of a ring $R$ in an overring $T$ is defined to be the smallest subring $S$ of $T$ containing $R$ that is \emph{division closed} in $T$, i.e., such that $x\in S$ is invertible in $S$ if it is invertible in $T$.
Furthermore, a group $\Gamma$ is \emph{locally indicable} if it is either trivial or every non-trivial finitely generated subgroup $H$ of $\Gamma$ admits a surjection onto $\ZZ$.
The groups considered in \cref{theo:main_pseudosylvester} are all locally indicable.

Observe that if $H$ is a non-trivial finitely generated subgroup of a locally indicable group $\Gamma$, $N$ is the kernel of a surjection $H\to \ZZ$ and $t\in H$ maps to a generator of $\ZZ$, then the powers of $\tilde t$ in $E\ast \Gamma$ are $E\ast N$-independent. A division $E\ast\Gamma$-ring of fractions is called Hughes-free if it ``preserves'' this property.

\begin{defi}
   A division $E\ast \Gamma$-ring of fractions $E\ast \Gamma\hookrightarrow \DC$ is \emph{Hughes-free} if, for every non-trivial finitely generated subgroup $H$ of $\Gamma$, every normal subgroup $N\trianglelefteq H$ such that $H/N$ is infinite cyclic and every $t\in H$ projecting to a generator of the quotient, the powers of $\tilde t$ are $\DC_N$-linearly independent, where $\DC_N$ denotes the division closure of $E\ast N$ in $\DC$.
\end{defi}

As proved by Hughes in \cite{Hughes1970}, if there exists a Hughes-free division $E\ast \Gamma$-ring of fractions, then it is unique up to $E\ast \Gamma$-isomorphism (see also \cite{Sanchez2008}*{Hughes Theorem I} for a detailed proof of this result).
Whereas the existence of a Hughes-free division ring for a general locally indicable group has only been settled for group rings $\Gamma$ when $K$ has characteristic zero (\cite{JL2020}), it is known for crossed products with the groups considered in \cref{theo:main_pseudosylvester}:

\begin{theo}\label{theo:hughes_free_exists}
   Let $G$ be a group obtained as an extension
   \[1 \rightarrow F \rightarrow G \rightarrow \ZZ \rightarrow 1\]
   where $F$ is a free group. Then, for every division ring $E$ and every crossed product $E\ast G$, there exists a Hughes-free division $E\ast G$-ring of fractions $\DC$. Moreover, if there exists a universal division $E\ast G$-ring of fractions, then it is isomorphic to $\DC$.
\end{theo}

\begin{proof}
    Every crossed product of a division ring with a free group admits a Hughes-free division ring of fractions (cf.\ \cite{Lewin1974}*{Proposition~6}, \cite{Sanchez2008}*{Example~5.6(e)}). Therefore, we obtain by a result of Hughes (cf.\ \cite{Sanchez2008}*{Hughes' Theorem II}) that $E\ast G$ admits a Hughes-free division ring of fractions. The final statement can be proved either by applying \cite{Sanchez2008}*{Example~6.19 \& Proposition~6.23} to the subnormal series $1\trianglelefteq F\trianglelefteq G$ or using \cite{JL2020}*{Corollary 8.2}.
\end{proof}

Thus, in our particular setting, $\DC_{E\ast G}$ is Hughes-free, and we are going to use the uniqueness of the Hughes-free division ring of fractions to describe two concrete realizations of $\DC_{E\ast G}$.

\subsection{\texorpdfstring{$\DC_{E\ast G}$}{D\_E*G} and the Malcev--Neumann construction} \label{subsubsect:Malcev_Neumann}

Brodskii (\cite{Brodskii1984}) proved that the family of locally indicable groups coincides with the family of left orderable groups admitting a \emph{Conradian order}, i.e., a left order with the property that for any $e< g,h \in \Gamma$, we have $h< gh^2$ (cf.\ \cite{Sanchez2008}*{Proposition 2.31}).
Even for a general left order $\leq$ on a group $\Gamma$, one can consider the set $E((\Gamma,\leq))$ of formal power series
\[ x = \sum_{g\in \Gamma} \tilde{g}\lambda_g, \textrm{ with } \lambda_g\in E,
\]
whose support $\supp(x) = \{g\in \Gamma\mid \lambda_g\ne 0\}$ is well-ordered with respect to $\leq$.

Malcev (\cite{Malcev1948}) and Neumann (\cite{Neumann1949}) proved independently that, if $\leq$ is also compatible with the right multiplication, then the natural sum and product of series are well-defined, and $K((\Gamma,\leq))$ for a field $K$ is a division ring in which $K\Gamma$ embeds. In general, $E((\Gamma,\leq))$ is not a ring, but it is still a right $E$-vector space, and we can see $E\ast \Gamma$ as a subring of $\Endd(E((\Gamma,\leq)))$ by extending the (left) multiplication defined in $E\ast \Gamma$ (cf.\ \cite{Grater2019}*{Section~7} for a detailed explanation and further properties of this embedding).
For the intermediate case of a Conradian order, Gräter has recently provided a partial generalization of the Malcev--Neumann result assuming the existence of a Hughes-free division ring of fractions:

\begin{theo}[\cite{Grater2019}*{Theorem~8.1 \& Corollary~8.3}] \label{theo:dubrovin_hughes_free}
  Let $E$ be a division ring and $\Gamma$ a locally indicable group. If there exists a Hughes-free division $E\ast \Gamma$-ring of fractions, then it is isomorphic to the division closure of $E\ast \Gamma$ in $\Endd(E((\Gamma,\leq)))$, where $\leq$ is any Conradian order on $\Gamma$.
\end{theo}

Combined with \cref{theo:hughes_free_exists}, we obtain:
\begin{coro}
    The universal division $E\ast G$-ring of fractions $\DC_{E\ast G}$ can be realized as the division closure of $E\ast G$ in $\Endd(E((G,\leq)))$ for any Conradian order $\leq$ on $G$.
\end{coro}

\subsection{\texorpdfstring{$\DC_{KG}$}{D\_KG} and the Atiyah conjecture} \label{subsubsect:Atiyah}

Let $\Gamma$ be a group and denote by $\ell^2(\Gamma)$ the $\ell^2$-Hilbert space with orthonormal basis $\Gamma$, i.e., the space whose elements are square-summable series $\sum_{g\in \Gamma} \lambda_g g$ with complex coefficients equipped with the standard $\ell^2$-scalar product.
The \emph{group von Neumann algebra} $\NC(\Gamma)$ is the algebra of bounded $\Gamma$-equivariant operators $T\colon\ell^2(\Gamma)\to\ell^2(\Gamma)$, where $\Gamma$ acts on $\ell^2(\Gamma)$ by left multiplication.
For every subfield $K$ of $\CC$, we can consider $K\Gamma$ as a subring of the group von Neumann algebra $\mathcal N(\Gamma)$ by identifying an element $a$ with the bounded $\Gamma$-equivariant operator $r_a\colon \ell^2(\Gamma)\to \ell^2(\Gamma)$ given by right multiplication by $a$.

The subset of (all) non-zero-divisors in the ring $\mathcal N(\Gamma)$ is right Ore, and the Ore localization, the algebra of unbounded affiliated operators $\mathcal U(\Gamma)$, is a von Neumann regular ring (cf.\ \cite{Luck2002}*{Theorem~8.22}).
The algebraic structure of the ring $\mathcal U(\Gamma)$ plays an important role for the strong Atiyah conjecture on the integrality of $L^2$-Betti numbers for torsion-free groups, which can equivalently be formulated as follows (cf.\ \cite{Luck2002}*{Lemma~10.39}):

\begin{defi}\label{conj:strong_Atiyah}
  A torsion-free group $\Gamma$ is said to satisfy the \emph{strong Atiyah conjecture} over the subfield $K$ of $\CC$ if the division closure $\DC(G;K)$ of $KG$ in $\mathcal U(G)$ is a division ring.
\end{defi}

There is no known example of a torsion-free group that does not satisfy the strong Atiyah conjecture. It is known to hold for locally indicable groups (\cite{JL2020}*{Theorem~1.1}), and the first proof covering the class of groups in \cref{theo:main_pseudosylvester} goes back to Linnell (\cite{Linnell1993}, see also \cite{Luck2002}*{Chapter~10}).

In Linnell's proof, Hughes-freeness was already used to identify the $\CC F$-rings $\DC(F;\CC)$ and $\DC_{\CC F}$, and the same arguments apply for any subfield $K$ of $\CC$. Using this, one can directly exhibit $\DC(G;K)$ as the Ore division ring of fractions of $\DC_{KF}\ast \ZZ$. Indeed, this crossed product can be built as a subring of $\UC(G)$, and hence, inasmuch as $\DC(G;K)$ is a division ring containing $\DC_{KF}\ast \ZZ$, it also contains the ring $\Ore(\DC_{KF} \ast \ZZ)$. Since the latter is a division subring containing $KG$, necessarily $\DC(G;K)=\Ore(\DC_{KF} \ast \ZZ)$. Using this, or alternatively  using \cite{JL2020}*{Corollary 6.3} together with \cref{theo:hughes_free_exists} we deduce the following.

\begin{coro}
    If $K$ is a subfield of $\CC$, then $\DC_{KG}$ agrees with the division closure $\DC(G;K)$ of $KG$ in $\mathcal U(G)$.
\end{coro}

\end{document}